\documentclass[12pt]{amsart}
\usepackage{amsmath,amssymb,latexsym, amsthm, amscd, mathrsfs, stmaryrd} %, txfonts}
\usepackage[all]{xy}

\usepackage[all]{xy}
\usepackage{color}
\usepackage{hyperref}

\theoremstyle{definition}
\newtheorem{rem}[section]{Remark}%[section]
\theoremstyle{plain}
\newtheorem{prop}[section]{Proposition}
\newtheorem{thm}[section]{Theorem}
\newtheorem{lem}[section]{Lemma}

\newcommand{\mbb}{\mathbb}
\newcommand{\mbf}{\mathbf}
\newcommand{\mrm}{\mathrm}

\renewcommand{\d}{\mbf d}
\newcommand{\e}{\mbf e}
\newcommand{\f}{\mbf f}
\renewcommand{\k}{\mbf k}

\newcommand{\dUi}{\dot {\mathbf U}^{\imath}}
\newcommand{\Ui}{\mathbf U^{\imath}}
\newcommand{\Uj}{\mathbf U^{\jmath}}
\newcommand{\Si}{\mbf S^{\imath}}
\newcommand{\Sj}{\mbf S^{\jmath}}
\renewcommand{\t}{t}
\newcommand{\bj}{\mbf j}

\title[Canonical bases of $\Ui(\mathfrak{sl}_2)$]{On canonical bases  of Letzter algebra $\Ui(\mathfrak{sl}_2)$}

\author[Yiqiang Li]{Yiqiang Li${}^{\dagger}$}
\thanks{$^{\dagger}$Partially supported by the NSF  grant DMS-1801915.}
\address{University at Buffalo \\ the State University of New York}
\email{yiqiang@buffalo.edu}


\begin{document}

\maketitle

%\begin{abstract}
%We show that  the geometrically-defined canonical basis of the coideal subalgebra $\Ui(\mathfrak{sl}_2)$ admits 
%the same explicit description as that of the algebraically-defined canonical basis, and hence that the two bases coincide.
%\end{abstract}

%\tableofcontents

\section*{introduction}
Let $\Ui\equiv\Ui(\mathfrak{sl}_2)$ be Letzter's coideal subalgebra of quantum $\mathfrak{sl}_2$ 
corresponding to the symmetric pair $(\mathfrak{sl}_2(\mbb C),\mbb C)$ (\cite{Le02}).
As a subalgebra of quantum $\mathfrak{sl}_2$, $\Ui$ is generated by the sum
$\mbf E + v\mbf K\mbf F+\mbf K$ of standard generators, and hence can be identified with the polynomial ring $\mbb Q(v)[t]$. 
In~\cite{BW13} and~\cite{LW18},  two distinguished bases, called $\imath$canonical bases, are constructed inside  the modified form of $\Ui$
via algebraic and geometric approaches respectively. 
The modified form of $\Ui$ can be  identified with  
a direct sum of two copies of   $\Ui\cong\mbb Q(v)[\t]$ itself.
An explicit and elegant formula, as a polynomial in $t$, of  algebraic  basis elements is conjectured in~\cite{BW13} and proved
in~\cite{BeW18}. 
The purpose of this short paper is to show that the geometric  basis in~\cite{LW18} admits the same description and, consequently,  that 
the two bases coincide.
Notice that the proofs are within the scope of {\it loc. cit.} and~\cite{BKLW}, whose notations shall be adopted here. 

\section*{The description}
%We shall use the notations in~\cite{LW18, BKLW} freely in this paper. 
Set 
$
 \llbracket n \rrbracket  = (v^n-v^{-n})/(v-v^{-1})$ and $\llbracket n\rrbracket !=\prod_{i=1}^n \llbracket i\rrbracket$.
Let $\Uj(\mathfrak{sl}_3)$ be an associative algebra over $\mbb Q(v)$ generated by $e, f, k, k^{-1}$ and subject to the following defining relations.
\begin{align*}
kk^{-1}=1,\
ke=v^3ek,\
kf=v^{-3} fk,\\
e^{2} f -\llbracket 2\rrbracket efe + f e^{2} = - \llbracket 2\rrbracket e (vk + v^{-1} k^{-1}),\\
f^2 e -\llbracket 2\rrbracket fef + ef^2 = - \llbracket 2\rrbracket (vk+v^{-1}k^{-1}) f.
\end{align*}
Let
$e^{(n)} = e^n/\llbracket n\rrbracket!$ and 
$f^{(n)}=f^n/\llbracket n\rrbracket!$.
By an induction argument and making use the above inhomogeneous Serre relations, we have the following formula in $\Uj(\mathfrak{sl}_3)$.

\begin{lem}
\label{a}
We have
$
f e^{(n+1)} = e^{(n)} ( fe - vef - \llbracket n \rrbracket (v^n k + v^{-n} k^{-1}) ) + v^{n+1} e^{(n+1)} f.
$
\end{lem}

Let $\Sj_{3,d}$ be $\Sj$ in~\cite[3.1]{BKLW} for $n=1$ and  $\e, \f, \k=\d_1\d_{2}^{-1}$ its generators. 
The assignments $e\mapsto \e$, $f\mapsto \f$ and $k^{\pm 1}\mapsto \k^{\pm 1}$ define an algebra homomorphism $\Uj(\mathfrak{sl}_3) \to \Sj_{3,d}$.
We set
\[
A_{a,b} =
\begin{bmatrix}
a & 0 & b\\
0 & 1 & 0 \\
b & 0 & a
\end{bmatrix}, \quad a, b\in \mbb N.
\]
Let
$
\bj_d= [A_{d,0}],
$
an idempotent in $\Sj_{3,d}$ ({\it loc. cit.} 3.17). We consider the subalgebra
$
\Si_{2,d}=\bj_d\Sj_{3,d}\bj_d
$
and its integral form $_{\mathcal A}\Si_{2,d}$ with $\mathcal A=\mbb Z[v, v^{-1}]$ (~\cite[4.1]{LW18}).
Let
\[
\mbf t_d=\left (\f \e + \frac{\k-\k^{-1}}{v-v^{-1}} \right ) \bj_d  \in \Si_{2,d}.
\]

\begin{lem}
\label{b}
In $\Si_{2,d}$, one has
$
\f^{(n)} \e^{(n)} \bj_{d} = ( \mbf t_d + \llbracket d -1\rrbracket ) (\mbf t_d + \llbracket d-3\rrbracket) \cdots (\mbf t_d + \llbracket d-  2n+1\rrbracket)/\llbracket n\rrbracket !.
$
\end{lem}

\begin{proof}
When $n=1$, this is the defining relation for $\mbf t_d$ in~\cite[Remark 5.3]{BKLW}.
Assume  the statement holds for $n$. 
Due to Lemma~\ref{a}, $\f\bj_d=0$ and $\k\bj_d= v^{1-d}$,
we have 
\begin{align*}
\f^{(n+1)} \e^{(n+1)}\bj_d 
& =
\frac{1}{\llbracket n+1\rrbracket} \f^{(n)} \f \e^{(n+1)} \bj_d
=
\frac{1}{\llbracket n+1\rrbracket}
\f^{(n)} \e^{(n)} ( \f \e - \llbracket n\rrbracket ( v^n v^{1-d} + v^{-n} v^{d-1})) \bj_d\\
& = 
%\frac{1}{\llbracket n+1\rrbracket}
\f^{(n)} \e^{(n)}  ( \mbf t_d + \llbracket d-1\rrbracket - \llbracket n\rrbracket ( v^n v^{1-d} + v^{-n} v^{d-1}))\bj_d/\llbracket n+1\rrbracket\\
& =  
%\frac{1}{\llbracket n+1\rrbracket}
\f^{(n)} \e^{(n)} (\mbf t_d + \llbracket (d-1)-2n\rrbracket)\bj_d/\llbracket n+1\rrbracket\\
& =
 ( \mbf t_d + \llbracket d -1\rrbracket ) (\mbf t_d + \llbracket d-3 \rrbracket) \cdots (\mbf t_d + \llbracket d - 2n-1\rrbracket)/\llbracket n+1\rrbracket !.
\end{align*}
Lemma follows by induction. 
\end{proof}

Lemma~\ref{b} provides a characterization of  $\Si_{2,d}$ as follows.

\begin{prop}
\label{defining}
The algebra $\Si_{2,d}$ is isomorphic to the quotient algebra of $\mbb Q(v)[\t]$ by the ideal generated by the polynomial
$
( \t + \llbracket d -1\rrbracket ) (\t + \llbracket d-3\rrbracket) \cdots (\t + \llbracket - d- 1\rrbracket).
$
\end{prop}

\begin{proof}
The map $\t\mapsto \mbf t_d$ defines a surjective algebra homomorphism $\phi^{\imath}_d: \mbb Q(v)[t] \to \Si_{2,d}$. 
Due to $\f^{(d+1)} \e^{(d+1)} \bj_d=0$ and Lemma~\ref{a}, the polynomial above is zero in $\Si_{2,d}$ for $\t=\mbf t_d$. 
So  $\phi^{\imath}_d$  factors through the desired quotient. 
Clearly the dimensions of $\Si_{2,d}$ and the quotient algebra are the same, so they must be isomorphic. 
The proposition is thus proved.
\end{proof}

Define an equivalence relation on the set $\{A_{a,b}| a, b\in \mbb N\}$ by $A_{a, b} \sim A_{a', b'}$ if $a\equiv a'$ (mod $2$) and $b=b'$.
Let $\check A_{a, b}$ be the equivalence class of $A_{a, b}$. As a $\mbb Q(v)$-vector space,
the modified form $\dUi$ of $\Ui$ 
is spanned by the canonical basis elements $b_{\check A_{0,d}}$ and $b_{\check A_{1,d}}$ for $d\in \mbb N$. 
Let 
$\dUi_{ 0} = \mrm{Span} \{ b_{\check A_{0,d}}, b_{\check A_{1,d+1}}| d\ \mbox{even}\}$ and  
$\dUi_{1} = \mrm{Span} \{ b_{\check A_{0,d}}, b_{\check A_{1,d-1}}| d\ \mbox{odd}\}.$
Then $\dUi =\dUi_{ 0} \oplus \dUi_{1}$ as algebras.
We have an isomorphism $\dUi_{0} \to \mbb Q(v) [\t]$  (resp. $\dUi_{ 1}\to \mbb Q(v)[\t]$) via $b_{\check A_{1,1}} \mapsto \t$ 
(resp. $b_{\check A_{0, 1}}\mapsto \t$). 
%So we can, and will,  identify $\dUi$ with the direct sum of two copies of $\mbb Q(v)[t]$. 
The isomorphisms are  compatible with $\phi^{\imath}_d$.
We have the following explicit description\footnote{We thank Weiqiang Wang for pointing out a gap in an early draft, and helpful comments.} 
of  geometrically-defined canonical basis elements of $\dUi$.

\begin{thm}
\label{desc}
The canonical basis elements of $\dUi\equiv\dUi(\mathfrak{sl}_2)$  in~ \cite{LW18} are  of the form
\begin{align}
\label{c}
b_{\check A_{0,d}}= 
\frac{( \t + \llbracket d -1\rrbracket ) (\t + \llbracket d-3\rrbracket) \cdots (\t + \llbracket -d+3\rrbracket) ( \t + \llbracket -d+1\rrbracket)}{\llbracket d\rrbracket !},
\ \forall d\in \mbb N;\\
\label{d}
b_{\check A_{1,d+1}}
= 
%\frac{\t}{\llbracket d+1\rrbracket} \cdot
\frac{t \cdot ( \t + \llbracket d -1\rrbracket ) (\t + \llbracket d-3\rrbracket) \cdots (\t + \llbracket -d+3\rrbracket) ( \t +\llbracket -d+1\rrbracket)}{\llbracket d+1\rrbracket !},
\ \forall d\in \mbb N.
\end{align}
\end{thm}

\begin{proof}
Let us denote the polynomial in (\ref{c}) by $P_{0,d}(\t)$ and that in (\ref{d}) by $P_{1,d+1}(\t)$.
By Lemma~\ref{b}, we have $\f^{(d)} \e^{(d)} \bj_d=P_{0, d}(\mbf t_d)$.
Observe that the element $\f^{(d)} \e^{(d)} \bj_d$ is a canonical basis element in $\Si_{2,d}$, corresponding to the constant sheaf on the product of maximal isotropic 
Grassmannians. 
Indeed, by~\cite[Thm. 3.7(a)]{BKLW}, we have
$
\f^{(d)} \e^{(d)} \bj_d =\sum_{i=0}^d v^{-\binom{i}{2}} [A_{i, d-i}]
.
$
By~\cite[Prop. 6.3]{LW18}, the canonical basis elements of $\dUi$ get sent to  canonical basis elements in $\Si_{2,d}$  or zero
via $\phi^{\imath}_d$.
So $b_{\check A_{0,d}}=P_{0,d}(\t)$, and (\ref{c}) holds.

Now consider the element $P_{1,d+1}(\mbf t_{d+2})$ in $\Si_{2, d+2}$. 
By rewriting the factor  $t$ in $P_{1, d+1}(t)$ as $(t+\llbracket d+1\rrbracket) - \llbracket d+1\rrbracket$ and combining with the remaining terms, there is
\begin{align*}
P_{1,d+1} (\mbf t_{d+2}) = \f^{(d+1)} \e^{(d+1)} \bj_{d+2} - P_{0, d} (\mbf t_{d+2}).
\end{align*}
Hence $P_{1,d+1}(\mbf t_{d+2})$ is in the integral form ${}_{\mathcal A} \Si_{2, d}$.
Further, the polynomial $P_{1, d+1}(\mbf t_{d+2})$ can be rewritten as
\begin{align*}
P_{1,d+1} (\mbf t_{d+2}) =P_{0,d} (\mbf t_{d+2})+
\frac{
( \mbf t_{d+2} + \llbracket d -1\rrbracket ) (\mbf t_{d+2} + \llbracket d-3\rrbracket) \cdots (\mbf t_{d+2} + \llbracket - d- 1\rrbracket)
}{\llbracket d+1\rrbracket!}.
\end{align*}
With the above expression, Proposition~\ref{defining} and the property of the transfer map $\phi^{\imath}_{d+2,d}$ in~\cite[(6.6)]{LW18} and~\cite{FL19}, we must have 
\begin{align}
\label{transfer}
\phi^{\imath}_{d+2,d} (P_{1,d+1}(\mbf t_{d+2})) = P_{0, d}(\mbf t_d). 
\end{align}
By the definition of $P_{1,d+1}(\t)$, there is $P_{1,d+1}(\mbf t_{d+2}) = \frac{1}{\llbracket d+1\rrbracket} \mbf t_{d+2} * \{A_{2,d}\}$ and so, in light of
the positivity property in~\cite[Theorem 6.12]{LW18}, it leads to
\[
P_{1,d+1}(\mbf t_{d+2}) =\{A_{1,d+1}\} + \sum_{i>1} c_i \{ A_{i, d+2-i}\}, \quad c_i\in \mbb N[v, v^{-1}].
\]
The positivity of $c_i$ together with (\ref{transfer})  and {\it loc. cit.} (6.6) implies that  $c_i=0$ and so
\[
P_{1,d+1}(\mbf t_{d+2}) =\{ A_{1, d+1}\} \in \Si_{2,d+2}.
\]
Therefore, we have $b_{\check A_{1, d+1}} =P_{1,d+1}(\t)$ and (\ref{d}) holds. 
The proof is finished.
\end{proof}

We conclude the paper with a remark.

\begin{rem}
(1) The canonical basis of $\Si_{2,d}$ is 
$\{
P_{0, d-2i}(\mbf t_d), P_{1, d-2i-1}(\mbf t_d) | 0\leq i \leq \lfloor d/2\rfloor 
\}
$.

(2) Clearly
$
b_{\check A_{0, 1}} * b_{\check A_{0, d}} = \llbracket d+1\rrbracket b_{\check A_{1, d+1}},
b_{\check A_{0, 1}} * b_{\check A_{1, d+1}} = \llbracket d+2\rrbracket b_{\check A_{0, d+2}} + \llbracket d+1\rrbracket b_{\check A_{0, d}}.
$

(3) The $\Ui$ in~\cite{BW13} and~\cite{LW18} differs by an involution 
$(\mbf E, \mbf F, \mbf K)\mapsto (\mbf F, \mbf E,\mbf K^{-1})$ on quantum $\mathfrak{sl}_2$.
Hence the canonical bases constructed therein coincide.
%(3) According to~\cite{FL19}, $\Ui \cong \mbb Q(v)[t]$ with 
%$t \mapsto \mbf E + v\mbf K \mbf F +\mbf K$ where $\mbf E, \mbf F, \mbf K^{\pm 1}$ are standard generators of quantum
%$\mathfrak{sl}_2$.
\end{rem}

\end{document}